\documentclass[journal]{IEEEtran}  % Comment this line out if you need a4paper

\usepackage{amsfonts,amssymb}
 
\usepackage{amsthm}
\usepackage{amsmath}

\newtheorem{lemma}{Lemma}
\newtheorem{theorem}{Theorem}
\newtheorem{definition}{Definition}

\newtheorem{assumption}{Assumption}

\newtheorem{remark}{Remark}
\usepackage{cite}
\usepackage{algorithmic}
\usepackage{textcomp}
\usepackage{graphicx}
\usepackage{subfig}
\usepackage{courier}
\usepackage{mathtools}
\usepackage{float}
\usepackage{import}
\usepackage{lipsum}
\usepackage{pgfplots}
\usepackage[ruled,norelsize]{algorithm2e}
\DeclarePairedDelimiter{\norm}{\lVert}{\rVert}
\newcommand\bovermat[2]{%
	\makebox[0pt][l]{$\smash{\overbrace{\phantom{%
					\begin{matrix}#2\end{matrix}}}^{\text{#1}}}$}#2}

\usepackage{tikz}
\usetikzlibrary{arrows.meta}
\tikzset{>={Latex[width=1.25mm,length=1.25mm]}}

\def\BibTeX{{\rm B\kern-.05em{\sc i\kern-.025em b}\kern-.08em
		T\kern-.1667em\lower.7ex\hbox{E}\kern-.125emX}}
\usepackage[font=small,skip=8pt]{caption}
\usepackage{mathrsfs}
\usepackage{cuted}

\usepackage{accents}

\DeclareMathOperator*{\argmin}{arg\,min}

\makeatletter
\let\NAT@parse\undefined
\newcommand{\removelatexerror}{\let\@latex@error\@gobble}
\makeatother
\usepackage{hyperref}  %hyperref still needs to be put at the end!

\title{\LARGE \bf
Data-based Control of Feedback Linearizable Systems
}

\author{Mohammad Alsalti$^{1}$, Victor G. Lopez$^{1}$, Julian Berberich$^{2}$, Frank Allgöwer$^{2}$, and Matthias A. Müller$^{1}$ %
	\thanks{$^{1}$Leibniz University Hannover, Institute of Automatic Control, 30167 Hannover, Germany. E-mail:\{\href{maitlo:alsalti@irt.uni-hannover.de}{alsalti},\href{maitlo:lopez@irt.uni-hannover.de}{lopez},\href{maitlo:mueller@irt.uni-hannover.de}{mueller}\}@irt.uni-hannover.de}%
	\thanks{$^{2}$University of Stuttgart, Institute for Systems Theory and Automatic Control, 70550 Stuttgart, Germany. E-mail: \{\href{maitlo:julian.berberich@ist.uni-stuttgart.de}{julian.berberich}, \href{maitlo:frank.allgower@ist.uni-stuttgart.de}{frank.allgower}\}@ist.uni-stuttgart.de}%
	\thanks{This work has received funding from the European Research Council (ERC) under the European Union’s
		Horizon 2020 research and innovation programme (grant agreement No 948679). This work was also funded by Deutsche Forschungsgemeinschaft (DFG, German Research Foundation) under Germany’s Excellence Strategy - EXC 2075 - 390740016 and under grant 468094890. We acknowledge the support by the Stuttgart Center for Simulation Science (SimTech).}
}

\newcommand\copyrighttext{%
	\footnotesize \copyright 2023 IEEE. Personal use of this material is permitted. Permission from IEEE must be obtained for all other uses, in any current or future media, including reprinting/republishing this material for advertising or promotional purposes, creating new collective works, for resale or redistribution to servers or lists, or reuse of any copyrighted component of this work in other works.}
\newcommand\copyrightnotice{%
	\begin{tikzpicture}[remember picture,overlay]
		\node[anchor=south,yshift=3pt] at (current page.south) {\fbox{\parbox{\dimexpr\textwidth-\fboxsep-\fboxrule\relax}{\copyrighttext}}};
	\end{tikzpicture}%
}

\begin{document}
	\maketitle
	\thispagestyle{empty}
	\pagestyle{empty}
	\copyrightnotice
	\begin{abstract}
	We present an extension of Willems' Fundamental Lemma to the class of multi-input multi-output discrete-time feedback linearizable nonlinear systems, thus providing a data-based representation of their input-output trajectories. Two sources of uncertainty are considered. First, the unknown linearizing input is inexactly approximated by a set of basis functions. Second, the measured output data is contaminated by additive noise. Further, we propose an approach to approximate the solution of the data-based simulation and output matching problems, and show that the difference from the true solution is bounded. Finally, the results are illustrated on an example of a fully-actuated double inverted pendulum.
	\end{abstract}

	\section{Introduction}
	Over the past two decades, researchers explored designing controllers directly from data without explicitly identifying a mathematical model of the system (cf. \cite{Hou13} and the references therein). In contrast to model-based control techniques, direct data-based control can be useful in cases where modeling complex systems from first principles is challenging \cite{Dorfler22_TAC}.\par
	A remarkable result from behavioral control theory \cite{Willems05} states that for a controllable, discrete-time linear time-invariant (DT-LTI) system, the entire vector space of input-output trajectories can be spanned by a single, persistently exciting, input-output trajectory. Now known as the \textit{fundamental lemma}, this result recently motivated a large number of works in the field of direct data-based system analysis and controller design. For example, it was used for data-based simulation and control of DT-LTI systems \cite{Rapisarda08}. It has also been translated to the state-space framework \cite{Berberich20, vanWaarde20} and was used to design {LQR} controllers \cite{Lopez21, Persis20, Florian21}, as well as predictive controllers \cite{Coulson20, Coulson21} with stability and robustness guarantees \cite{Berberich203, Berberich204}. Extensions to Hammerstein-Wiener systems appeared in \cite{Berberich20} and Second-Order Volterra systems in \cite{Escobedo20}. It was also used to design controllers for classes of nonlinear systems purely from data in \cite{Bisoffi20, Luppi21, Guo21, DePersis22}. For a more complete and comprehensive review, the reader is referred to the review paper \cite{Markovsky21}. Apart from the fundamental lemma, data-driven stabilization for single-input single-output (SISO) feedback linearizable nonlinear systems appeared in \cite{Tabuada17}. There, input-affine continuous-time systems were addressed assuming constant inter-sampling behavior of the states under high enough sampling rate.\par
	In this paper, we build on our previous work \cite{AlsaltiBerLopAll2021}, where SISO flat systems were investigated and a data-based system representation was given assuming that an exact expansion using basis functions is known. Compared to \cite{AlsaltiBerLopAll2021}, here we treat general multi-input multi-output (MIMO) full-state feedback linearizable nonlinear systems, which model a variety of physical systems, e.g., robotic manipulators (other examples can be found in \cite{Murray95differentialflatness}). Furthermore, we provide suitable error bounds for the results of the simulation and output matching control problems when \textbf{(i)} the given basis functions do not exactly represent the unknown nonlinearities, and \textbf{(ii)} the measured output data is noisy.\par
	After reviewing notation, definitions and existing results in Section \ref{prel}, we provide a data-based representation of feedback linearizable nonlinear systems in Section \ref{DB_analysis_sec} by exploiting linearity in transformed coordinates along with a set of basis functions that depend only on input-output data to approximate the unknown nonlinearities. In Section \ref{sim_om_sec}, we provide constructive methods to approximately solve the simulation and output-matching control problems despite basis functions approximation error and only using input and noisy output data. We then show that the difference between the estimated and true outputs is upper bounded. We illustrate the results on a model of a fully-actuated double inverted pendulum in Section \ref{examples} and conclude the paper in Section \ref{conclusion_sec}.\par
	\section{Preliminaries}\label{prel}
	\subsection{Notation}
	The set of integers in the interval $[a,b]$ is denoted by $\mathbb{Z}_{[a,b]}$. For a vector $w\in\mathbb{R}^n$, $p-$norms for $p=1,2,\infty$ are denoted by $\norm*{w}_p$, respectively, whereas $\norm*{M}_i$ for $i=1,2,\infty$ denotes the induced norm of a matrix $M$. We use $\mathbf{0}$ to denote a vector or a matrix of zeros of appropriate dimensions. An $n\times n$ identity matrix is denoted by $I_n$.\par
	For a sequence $\{\mathbf{z}_k\}_{k=0}^{N-1}$ with $\mathbf{z}_k\in\mathbb{R}^\eta$, each element is expressed as $\mathbf{z}_k=\begin{bmatrix}
	z_{1,k} & z_{2,k} & \dots & z_{\eta,k}
	\end{bmatrix}^\top$. The stacked vector of that sequence is given by $\mathbf{z}=\begin{bmatrix}
	\mathbf{z}_0^\top & \dots & \mathbf{z}_{N-1}^\top
	\end{bmatrix}^\top$, and a window of it by $\mathbf{z}_{[a,b]}=\begin{bmatrix}
	\mathbf{z}_a^\top & \dots & \mathbf{z}_b^\top
	\end{bmatrix}^\top$. The Hankel matrix of depth $L$ of this sequence is given by
	\begin{equation*}
		\begin{aligned}
			H_L(\mathbf{z})&=\begin{bmatrix}
				\mathbf{z}_0 & \mathbf{z}_1 & \dots & \mathbf{z}_{N-L}\\
				\mathbf{z}_1 & \mathbf{z}_2 & \dots & \mathbf{z}_{N-L+1}\\
				\vdots & \vdots & \ddots & \vdots\\
				\mathbf{z}_{L-1} & \mathbf{z}_L & \dots & \mathbf{z}_{N-1}
			\end{bmatrix}.
		\end{aligned}
	\end{equation*}
	\indent Throughout the paper, the notion of persistency of excitation (PE) is defined as follows.
	\begin{definition}
		The sequence $\{\mathbf{z}_k\}_{k=0}^{N-1}$ is said to be persistently exciting of order $L$ if \textup{rank}$\left(H_L(\mathbf{z})\right)=\eta L$.
	\end{definition}
%%%%%%%%%%%%%%%
	\subsection{Discrete-time feedback linearizable systems}
	Consider the following DT-MIMO square nonlinear system
	\begin{equation}
		\begin{matrix}
			\mathbf{x}_{k+1} = \boldsymbol{f}(\mathbf{x}_k,\mathbf{u}_k),&\quad&
			\mathbf{y}_k = \boldsymbol{h}(\mathbf{x}_k),
		\end{matrix}
		\label{NLsys}
	\end{equation}
	where $\mathbf{x}_k\in\mathbb{R}^n$ is the state vector and $\mathbf{u}_k,\mathbf{y}_k\in\mathbb{R}^m$ are the input and output vectors, respectively. The functions $\boldsymbol{f}:\mathbb{R}^n\times\mathbb{R}^m\to\mathbb{R}^n$, $\boldsymbol{h}:\mathbb{R}^n\to\mathbb{R}^m$ are analytic functions with $\boldsymbol{f}(\mathbf{0},\mathbf{0})=\mathbf{0}$ and $\boldsymbol{h}(\mathbf{0})=\mathbf{0}$. We define $\boldsymbol{f}_O^j$ as the $j-$th iterated composition of the undriven dynamics $\boldsymbol{f}_O\coloneqq\boldsymbol{f}(\cdot,\mathbf{0})$. A sequence $\{\mathbf{u}_k,\mathbf{y}_k\}_{k=0}^{N-1}$ is said to be an input-output trajectory of the nonlinear system \eqref{NLsys} if there exists an initial condition $\mathbf{x}_0$ such that \eqref{NLsys} holds for all $k\in\mathbb{Z}_{[0,N-1]}$.\par
	As defined in \cite{MonacoNor1987}, each output $y_i=h_i(\mathbf{x})$ of the nonlinear system \eqref{NLsys}, for $i\in\mathbb{Z}_{[1,m]}$, is said to have a (globally) well-defined relative degree $d_i$ if at least one of the $m$ inputs at time $k$ affects the $i-$th output at time $k+d_i$. In particular,
	\begin{equation}
	\begin{aligned}
		y_{i,k+d_i} &= h_i(\boldsymbol{f}_O^{d_i-1}\left(\boldsymbol{f}(\mathbf{x}_k,\mathbf{u}_k)\right)).
		\label{output_k}
	\end{aligned}
	\end{equation}
	\indent To bring the system in \eqref{NLsys} to the DT normal form, we make the following standard assumptions (see \cite{MonacoNor1987}).
		\begin{assumption}\label{MIMO_img_assmp}
			For any $\mathbf{x}_k\in\mathbb{R}^n$, $\exists\,\tilde{\mathbf{u}}_k\in\mathbb{R}^m$ such that
			\begin{equation}
				h_i(\boldsymbol{f}_O^{d_i-1}\left(\boldsymbol{f}(\mathbf{x}_k,\tilde{\mathbf{u}}_k)\right))=0,\qquad \forall i\in\mathbb{Z}_{[1,m]}.
			\end{equation}
		\end{assumption}
		\begin{assumption}\label{D_full_rank_assmp}
			For any $\mathbf{x}_k\in\mathbb{R}^n$, the decoupling matrix $\mathcal{D}(\mathbf{x}_k,\mathbf{u}_k)$  evaluated at $\tilde{\mathbf{u}}_k$ has rank$\left(\mathcal{D}(\mathbf{x}_k,\tilde{\mathbf{u}}_k)\right) = m$, where
			\begin{equation}
				\mathcal{D}(\mathbf{x}_k,{\mathbf{u}}_k)_{\{i,j\}} ={\partial  h_i(\boldsymbol{f}_O^{d_i-1}\left(\boldsymbol{f}(\mathbf{x}_k,{\mathbf{u}}_k)\right))}\big/{\partial {u}_{j,k}}.
			\end{equation}
		\end{assumption}
		\begin{assumption}\label{sum_di_asmp}
			The sum of relative degrees of the outputs of system \eqref{NLsys} is equal to the system dimension, i.e., $\sum_i d_i=n$.
	\end{assumption}
	Assumptions \ref{MIMO_img_assmp} and \ref{D_full_rank_assmp} are standard and they are needed to invoke the implicit function theorem and show existence of an invertible coordinate transformation and a feedback linearizing control law $\mathbf{u}_k=\gamma(\mathbf{x}_k,\mathbf{v}_k)$ that results in a linear map from $\mathbf{v}_k$ to $\mathbf{y}_k$ (cf. \cite[Prop. 3.1]{MonacoNor1987}). Assumption \ref{sum_di_asmp} is needed for the system to be full-state feedback linearizable, i.e., having no internal dynamics. For globally well-defined relative degrees, this condition can be checked by perturbing the system from rest and recording the first time instants at which each output changes from zero (cf. \eqref{output_k}). If the sum of all these instances is $n$, Assumption \ref{sum_di_asmp} is fulfilled. Feedback linearization of DT-MIMO nonlinear systems is formally stated as follows.
	\begin{theorem}[\hspace{-0.25mm}\cite{MonacoNor1987}, Prop. 3.1]\label{MIMO_FL_thm}
		Let Assumptions \ref{MIMO_img_assmp}-\ref{sum_di_asmp} be satisfied, then there exists an invertible (w.r.t. $\mathbf{v}_k$) feedback control law $\mathbf{u}_k=\gamma(\mathbf{x}_k,\mathbf{v}_k)$, with $\gamma:\mathbb{R}^n\times\mathbb{R}^m\to\mathbb{R}^m$ and an invertible coordinate transformation $\Xi_k=T(\mathbf{x}_k)$, such that system \eqref{NLsys} is input-output decoupled and can be written as
		\begin{equation}
			\begin{matrix}
				\Xi_{k+1} = \mathcal{A}\Xi_k + \mathcal{B}\mathbf{v}_k,&\quad
				\mathbf{y}_k = \mathcal{C}\Xi_k,
			\end{matrix}
			\label{BINF}%
		\end{equation}%
		where $\Xi_k=\begin{bmatrix}\xi_{1,k}&\dots&\xi_{n,k}\end{bmatrix}^\top\in\mathbb{R}^{n}$ is defined as
		\begin{equation}
			\begin{aligned}
				\Xi_k &=\begin{bmatrix}
					y_{1,[k,k+d_1-1]}^\top & \dots & y_{m,[k,k+d_m-1]}^\top
				\end{bmatrix}^\top.
			\end{aligned}
			\label{Xi}
		\end{equation}
		Further, $\mathcal{A},\,\mathcal{B},\,\mathcal{C}$ are in block-Brunovsky\footnote{We use the term ``block-Brunovsky form" to refer to block diagonal matrices with Brunovsky canonical form matrices as diagonal elements. See Appendix \ref{Bruno_app} for more information on the structure of the $\mathcal{A,B,C}$ matrices.} form, which are a controllable/observable triplet.
	\end{theorem}
	\indent Theorem \ref{MIMO_FL_thm} shows that a nonlinear system \eqref{NLsys} that satisfies Assumptions \ref{MIMO_img_assmp}-\ref{sum_di_asmp} has an equivalent linear representation where the input-state and input-output maps are linear and decoupled. That is, each synthetic input $v_{i}$ only affects its corresponding output $y_{i}$ for all $i\in\mathbb{Z}_{[1,m]}$. In the next section, we use Theorem \ref{MIMO_FL_thm} to provide a data-based representation of trajectories of full-state feedback linearizable systems.
	\section{Willems' Fundamental Lemma for Feedback Linearizable Systems}\label{DB_analysis_sec}
	{In the setting of data-based control, one typically only has access to input-output data and not to the synthetic input $\mathbf{v}_k$ or the corresponding state transformation $T(\mathbf{x}_k)$.} In order to come up with a data-based description of the trajectories of the nonlinear system \eqref{NLsys} (or the equivalent system \eqref{BINF}), the synthetic input is expressed as (cf. \cite{MonacoNor1987})
	\begin{equation}
		\begin{aligned}
			\mathbf{v}_k \hspace{-1mm}= \hspace{-1mm}\begin{bmatrix}
				v_{1,k} \\ \vdots \\ v_{m,k}
			\end{bmatrix} \hspace{-1mm}=\hspace{-1mm} \begin{bmatrix}
				\hspace{-0.5mm}{h}_1\hspace{-0.5mm}\left(\boldsymbol{f}_O^{d_1-1}\left(\boldsymbol{f}(\mathbf{x}_k,\mathbf{u}_k)\right)\right)\hspace{-0.5mm}\\
				\vdots\\
				\hspace{-0.5mm}{h}_m\hspace{-0.5mm}\left(\boldsymbol{f}_O^{d_m-1}\left(\boldsymbol{f}(\mathbf{x}_k,\mathbf{u}_k)\right)\right)\hspace{-0.5mm}
			\end{bmatrix}\hspace{-1mm}\eqqcolon\hspace{-1mm} \tilde{\Phi}(\mathbf{u}_k,\mathbf{x}_k).
		\end{aligned}
		\label{definition_of_v_and_Phi}
	\end{equation}
	By Theorem \ref{MIMO_FL_thm}, it holds that $\mathbf{x}_k = T^{-1}(\Xi_k)$ and, hence, one can define $\Phi(\mathbf{u}_k,\Xi_k)\coloneqq\tilde{\Phi}(\mathbf{u}_k,T^{-1}(\Xi_k))$, which allows us to parameterize $\mathbf{v}_k$ using input-output data only since $\Xi_k$ is given by shifted outputs (see \eqref{Xi}). Note that ${\Phi}(\mathbf{u}_k,\Xi_k)$ is unknown, therefore we approximate it by a set of basis functions that depend only on input and output data. In particular,
	\begin{align}
		\mathbf{v}_k & = \Phi(\mathbf{u}_k,\Xi_k)
		\label{basis}\\
		&= \hspace{-1mm}\begin{bmatrix}
				\phi_1(\mathbf{u}_k,\Xi_k)\\ \vdots \\ \phi_m(\mathbf{u}_k,\Xi_k)
		\end{bmatrix} \hspace{-1mm}=\hspace{-1mm} \begin{bmatrix}
			\rule[.5ex]{1em}{0.4pt} \,g_1^\top\, \rule[.5ex]{1em}{0.4pt}\\
			\vdots\\
			\rule[.5ex]{1em}{0.4pt} \,g_m^\top\, \rule[.5ex]{1em}{0.4pt}\\
		\end{bmatrix}\hspace{-1mm}\Psi(\mathbf{u},{{\Xi}_k})\hspace{-1mm} + \hspace{-1mm}\begin{bmatrix}
		\varepsilon_1(\mathbf{u}_k,{{\Xi}_k})\\ \vdots\\ \varepsilon_m(\mathbf{u}_k,{{\Xi}_k})
		\end{bmatrix}\notag\\
		&\eqqcolon \mathcal{G}\Psi(\mathbf{u}_k,{{\Xi}_k}) + \scalebox{1.5}{$\epsilon$}(\mathbf{u}_k,{{\Xi}_k}),\notag
	\end{align}
	where $\Psi(\mathbf{u}_k,{{\Xi}_k})$ is the vector of $r\in\mathbb{N}$ locally Lipschitz continuous and linearly independent basis functions $\psi_j:\mathbb{R}^m\times\mathbb{R}^n\to\mathbb{R}$ for $j\in\mathbb{Z}_{[1,r]}$ and $\scalebox{1.5}{$\epsilon$}(\mathbf{u}_k,{{\Xi}_k})$ is the stacked vector of approximation errors $\varepsilon_i:\mathbb{R}^m\times\mathbb{R}^n\to\mathbb{R}$ for $i\in\mathbb{Z}_{[1,m]}$. The term $\mathcal{G}\in\mathbb{R}^{m\times r}$ is the matrix of unknown coefficients of the basis functions and $g_i^\top$, for $i\in\mathbb{Z}_{[1,m]}$, represent its rows.
	For the theoretical analysis presented in this paper, we define, \emph{but never compute}, $\mathcal{G}$ as follows\footnote{The inner product is given by $\left< \rho_1, \rho_2 \right> = \int_\Omega \rho_1(\ell_1,\ell_2)\rho_2(\ell_1,\ell_2) d\ell_1 d\ell_2$ for some $\Omega\subset \mathbb{R}^m\times\mathbb{R}^n$.}
	\begin{equation}
		\mathcal{G}\coloneqq\argmin\limits_{{G}} \left\langle  \Phi- {G}\Psi, \Phi - {G}\Psi \right\rangle .\label{def_of_G_mat}
	\end{equation}
	\indent The minimization problem in \eqref{def_of_G_mat} is a least-squares problem that minimizes the average approximation error on a compact subset of the input-state space $\Omega\subset\mathbb{R}^m\times\mathbb{R}^n$.
	Furthermore, if the choice of basis functions contains $\Phi(\mathbf{u}_k,\Xi_k)$ in their span and the data is noiseless, then a unique solution to \eqref{def_of_G_mat} exists and the results of \cite{AlsaltiBerLopAll2021} are retrieved. For the subsequent analysis, the following assumption is made on $\mathcal{G}$.
	\begin{assumption}\label{rank_assmp_G}
		The matrix of coefficients $\mathcal{G}$ has full row rank.
	\end{assumption}
	From \eqref{output_k}, \eqref{definition_of_v_and_Phi} and \eqref{basis} (cf. also \eqref{y_for_proofs} below), it can be seen that Assumption~\ref{rank_assmp_G} corresponds to the outputs being linearly independent. This is fulfilled in, e.g., robotic manipulators and other fully-actuated mechanical systems (cf.~\cite{Murray95differentialflatness}).\par
	Inevitably, measured data is noisy. In what follows, we denote the collected output measurements by $\mathbf{\tilde{y}}_k=\mathbf{y}_k+\mathbf{w}_k$, where $\norm*{\mathbf{w}_k}_\infty\leq {w}^*$, for all $k\geq0$, is a uniformly bounded output measurement noise. As a result of using noisy data, the unknown nonlinear function $\Phi$ in \eqref{basis} is now expressed as
		\begin{equation}
			\Phi(\mathbf{u}_k,\Xi_k) = \mathcal{G}\Psi(\mathbf{u}_k,\tilde{\Xi}_k) + \scalebox{1.5}{$\epsilon$}(\mathbf{u}_k,\tilde{\Xi}_k) + \delta(\omega_k),
			\label{noisy_basis}
		\end{equation}
		where $\delta(\omega_k)\coloneqq \mathcal{G}\Psi(\mathbf{u}_k,{\Xi}_k) + \scalebox{1.5}{$\epsilon$}(\mathbf{u}_k,{\Xi}_k) - \mathcal{G}\Psi(\mathbf{u}_k,\tilde{\Xi}_k) - \scalebox{1.5}{$\epsilon$}(\mathbf{u}_k,\tilde{\Xi}_k)$, $\tilde{\Xi}_k=\Xi_k+\omega_k$ and $\omega_k=\begin{bmatrix}w_{1,[k,k+d_1-1]}^\top\,\cdots\, w_{m,[k,k+d_m-1]}^\top\end{bmatrix}^\top$. Now, one can substitute \eqref{noisy_basis} back into \eqref{BINF} to obtain
			\begin{align}
				\Xi_{k+1} &= \mathcal{A}\Xi_k + \mathcal{B}\mathcal{G}(\Psi(\mathbf{u}_k,{\tilde{\Xi}_k})+E(\mathbf{u}_k,{\tilde{\Xi}_k})+ D(\omega_k)),\notag\\
				{\mathbf{\tilde{y}}_k }&{= \mathcal{C}\Xi_k + \mathbf{w}_k},		\label{Lsys2}
			\end{align}
	where $E(\mathbf{u}_k,{\tilde{\Xi}_k})\coloneqq \mathcal{G}^\dagger\scalebox{1.5}{$\epsilon$}(\mathbf{u}_k,{\tilde{\Xi}_k})$, $D(\omega_k)=\mathcal{G}^\dagger\delta(\omega_k)$ and $\mathcal{G}^\dagger\coloneqq \mathcal{G}^\top(\mathcal{GG}^\top)^{-1}$ is a right inverse of $\mathcal{G}$, which exists by Assumption \ref{rank_assmp_G}. For convenience, we use the following notation throughout the paper
	\begin{equation}
			\begin{matrix}
			\hat{\Psi}_k(\mathbf{u},{\tilde{\Xi}}) \coloneqq \Psi(\mathbf{u}_k,{\tilde{\Xi}_k}), &
			\hat{\scalebox{1.5}{$\epsilon$}}_k(\mathbf{u},{\tilde{\Xi}}) \coloneqq 	\scalebox{1.5}{$\epsilon$}(\mathbf{u}_k,{\tilde{\Xi}_k}), \\
			\hat{\varepsilon}_{i,k}(\mathbf{u},{\tilde{\Xi}}) \coloneqq 	\varepsilon_i(\mathbf{u}_k,{\tilde{\Xi}_k}), &
			\hat{E}_k(\mathbf{u},{\tilde{\Xi}}) \coloneqq E(\mathbf{u}_k,{\tilde{\Xi}_k}),\\
			\hat{\delta}_k(\omega) \coloneqq \delta(\omega_k), & \hat{D}_k(\omega) \coloneqq D(\omega_k).
		\end{matrix}\label{imp_defs}
	\end{equation}
	Using the notation in \eqref{imp_defs}, system \eqref{Lsys2} can now be written as
		\begin{align}
			\Xi_{k+1} &= \mathcal{A}\Xi_k + \mathcal{B}\mathcal{G}(\hat{\Psi}_k(\mathbf{u},{\tilde{\Xi}})+\hat{E}_k(\mathbf{u},{\tilde{\Xi}}) + \hat{D}_k(\omega)),\notag\\
			{\mathbf{\tilde{y}}_k }&{= \mathcal{C}\Xi_k + \mathbf{w}_k},	\label{Lsys3}
		\end{align}%
	For the system in \eqref{Lsys3}, the following holds.
	\begin{lemma}\label{lemma_controllable_pair}
		The pair $(\mathcal{A,BG})$ is controllable.
	\end{lemma}
	\begin{proof}
		See Appendix \ref{app_controllable_pair}.
	\end{proof}%
	Systems that are in the block-Brunovsky canonical form as in \eqref{BINF} have two appealing properties that will be extensively used throughout the paper. First, the state $\Xi_k$, and its noisy counterpart $\tilde{\Xi}_k = \Xi_k+\omega_k$, are defined as the shifted outputs as in \eqref{Xi}. Second, the system is input-output decoupled from $\mathbf{v}_k$ to $\mathbf{y}_k$. This means that the $i-$th synthetic input at time $k$ is equal to the $i-$th (noiseless) output at time $k+d_i$, i.e.,
%	\begin{equation}
		\begin{align}
			\hspace{-1mm}y_{i,k+d_i}\hspace{-1mm} \stackrel{\eqref{output_k}, \eqref{definition_of_v_and_Phi}}{=}\hspace{-0.5mm} v_{i,k} \hspace{-0.5mm}&= \hspace{-0.5mm} \phi_i(\mathbf{u}_k,\Xi_k)\hspace{-1mm}\stackrel{\eqref{basis},\eqref{imp_defs}}{=}\hspace{-1mm} g_i^\top\hspace{-1mm}\left(\hspace{-0.5mm}\hat{\Psi}_k(\mathbf{u},{{\Xi}})\hspace{-1mm}+\hspace{-1mm}\hat{E}_k(\mathbf{u},{{\Xi}})\hspace{-0.5mm}\right)\hspace{-1mm}\notag\\
			&\stackrel{\eqref{noisy_basis}}{=}g_i^\top\hspace{-1.5mm}\left(\hat{\Psi}_k(\mathbf{u},{\tilde{\Xi}})\hspace{-1mm}+\hspace{-1mm}\hat{E}_k(\mathbf{u},{\tilde{\Xi}}) \hspace{-1mm}+\hspace{-1mm} \hat{D}_k(\omega)\right).\label{y_for_proofs}
		\end{align}
%	\end{equation}
	\indent In the following theorem, we extend the results of \cite{Willems05} to the class of DT-MIMO full-state feedback linearizable nonlinear systems. Theorem \ref{DB_MIMO_rep} studies the {nominal case} for which the approximation error in \eqref{basis} is zero and the data is noiseless. For this case, \eqref{Lsys3} reduces to
	\begin{equation}
		\Xi_{k+1} = \mathcal{A}\Xi_k + \mathcal{B}\mathcal{G}\hat{\Psi}_k(\mathbf{u},\Xi), \qquad \mathbf{y}_k = \mathcal{C}\Xi_k.\label{nominal_Lsys}
	\end{equation}
	The case for which the basis functions approximation errors and output noise are nonzero is studied in Section \ref{sim_om_sec}.
	\begin{theorem}\label{DB_MIMO_rep}
		Suppose Assumptions \ref{MIMO_img_assmp}--\ref{rank_assmp_G} are satisfied and let $\{\mathbf{u}_k\}_{k=0}^{N-1}$, $\{y_{i,k}\}_{k=0}^{N+d_i-1}$, for $i\in\mathbb{Z}_{[1,m]}$, be a trajectory of a full-state feedback linearizable system as in \eqref{NLsys}. Furthermore, let $\{\hat{\Psi}_k(\mathbf{u},\Xi)\}_{k=0}^{N-1}$ from \eqref{nominal_Lsys} be persistently exciting of order $L+n$. Then, any $\{\mathbf{\bar{u}}_k\}_{k=0}^{L-1}$, $\{\bar{y}_{i,k}\}_{k=0}^{L+d_i-1}$ is a trajectory of system \eqref{NLsys} if and only if there exists $\alpha\in\mathbb{R}^{N-L+1}$ such that the following holds
		\begin{equation}
			\hspace{-3mm}\begin{bmatrix}\hspace{-0.5mm}
				H_{L}(\hat{\Psi}(\mathbf{u},\Xi))\\ 	H_{L+1}(\Xi)\end{bmatrix}\hspace{-0.5mm}\alpha \hspace{-0.25mm}=\hspace{-0.25mm} \begin{bmatrix}
				\hat{\Psi}(\mathbf{\bar{u}},\bar{\Xi})\\ 	\bar{\Xi}
			\end{bmatrix}\hspace{-1mm},
			\label{inexact_result}
		\end{equation}
		where $\hat{\Psi}(\mathbf{\bar{u}},\bar{\Xi})$ is the stacked vector of the sequence $\{\hat{\Psi}_k(\mathbf{\bar{u}},\bar{\Xi})\}_{k=0}^{L-1}$, while $\Xi,\,\bar{\Xi}$ are the stacked vectors of $\{\Xi_k\}_{k=0}^{N},\,\{\bar{\Xi}_k\}_{k=0}^{L}$ which, according to \eqref{Xi}, are composed of $\{y_{i,k}\}_{k=0}^{N+d_i-1},\,\{\bar{y}_{i,k}\}_{k=0}^{L+d_i-1}$, respectively.
	\end{theorem}
	\begin{proof}
		According to Theorem \ref{MIMO_FL_thm}, $\{\hat{\Psi}_k(\mathbf{\bar{u}},\bar{\Xi})\}_{k=0}^{L-1}$, $\{\bar{y}_{i,k}\}_{k=0}^{L+d_i-1}$ is an input-output trajectory of \eqref{nominal_Lsys} if and only if $\{\mathbf{\bar{u}}_k\}_{k=0}^{L-1},\,\{\bar{y}_{i,k}\}_{k=0}^{L+d_i-1}$ is an input-output trajectory of \eqref{NLsys}.\par
		Using the input-output trajectory $\{\mathbf{u}_k\}_{k=0}^{N-1}$, $\{y_{i,k}\}_{k=0}^{N+d_i-1}$ of system \eqref{NLsys}, one can construct the following sequences $\{\hat{\Psi}_k(\mathbf{u},\Xi)\}_{k=0}^{N-1}$ and $\{\Xi_k\}_{k=0}^{N}$, which correspond to an input-state trajectory \eqref{nominal_Lsys}. Since $\{\hat{\Psi}_k(\mathbf{u},\Xi)\}_{k=0}^{N-1}$ is persistently exciting of order $L+n$ by assumption, and the pair $(\mathcal{A},\mathcal{B}\mathcal{G})$ is controllable by Lemma \ref{lemma_controllable_pair}, then \cite{Willems05} shows that any $\{\hat{\Psi}_k(\mathbf{\bar{u}},\bar{\Xi})\}_{k=0}^{L-1},\,\{\bar{\Xi}_{k}\}_{k=0}^{L-1}$ is an input-state trajectory of \eqref{nominal_Lsys} if and only if there exists $\alpha\in\mathbb{R}^{N-L+1}$ such that
		\begin{equation}
			\hspace{-3mm}\begin{bmatrix}\hspace{-0.5mm}
				H_{L}(\hat{\Psi}(\mathbf{u},\Xi))\\ H_{L}(\Xi_{[0,N-1]})\end{bmatrix}\hspace{-0.5mm}\alpha \hspace{-0.25mm}=\hspace{-0.25mm} \begin{bmatrix}
				\hat{\Psi}(\mathbf{\bar{u}},\bar{\Xi})\\ \bar{\Xi}_{[0,L-1]}
			\end{bmatrix}\hspace{-1mm}.
			\label{NL_incomplete_willems_lemma}
		\end{equation}
		Next, notice from \eqref{y_for_proofs} that for each output $\bar{y}_i$, for $i\in\mathbb{Z}_{[1,m]}$, the following holds
		\begin{align}
			\hspace{0mm}\bar{y}_{i,[L,L+d_i-1]}\label{last_di_outputs}&\hspace{2mm}\stackrel{\eqref{y_for_proofs}}{=} (I_{d_i}\otimes g_i^\top)\hat{\Psi}_{[L-d_i,L-1]}(\mathbf{\bar{u}},\bar{\Xi})\\
			&\hspace{2mm}\stackrel{\eqref{NL_incomplete_willems_lemma}}{=} (I_{d_i}\otimes g_i^\top) H_{d_i}(\hat{\Psi}_{[L-d_i,N-1]}(\mathbf{u},{\Xi}))\alpha\notag\\
			%\bar{y}_{i,[L,L+d_i-2]}
			&\hspace{2mm}\stackrel{\eqref{y_for_proofs}}{=} H_{d_i}\left(y_{i,[L,N+d_i-1]}\right)\alpha.\notag
		\end{align}
		Furthermore, the state $\bar{\Xi}_L$ can be written as
		\begin{align}
			\bar{\Xi}_L &\stackrel{\eqref{Xi}}{=} \begin{bmatrix}
				\bar{y}_{1,[L,L+d_1-1]} \\ \vdots \\ \bar{y}_{m,[L,L+d_m-1]}
			\end{bmatrix} \stackrel{\eqref{last_di_outputs}}{=} \begin{bmatrix}
			H_{d_1}(y_{1,[L,N+d_1-1]})\\ \vdots \\ H_{d_m}(y_{m,[L,N+d_m-1]})
		\end{bmatrix}\alpha\notag\\
		& = H_1(\Xi_{[L,N]})\alpha.\label{last_xi}
		\end{align}
	Finally, concatenating \eqref{last_xi} with \eqref{NL_incomplete_willems_lemma} results in \eqref{inexact_result} which completes the proof.
	\end{proof}
	Theorem \ref{DB_MIMO_rep} provides a purely data-based representation of full-state feedback linearizable systems. In particular, each input-output trajectory $\{\mathbf{\bar{u}}_k\}_{k=0}^{L-1},\{\bar{y}_{i,k}\}_{k=0}^{L+d_i-1}$, for $i\in\mathbb{Z}_{[1,m]}$, can be parameterized via \eqref{inexact_result} using a priori collected data $\{\mathbf{u}_k\}_{k=0}^{N-1},\{y_{i,k}\}_{k=0}^{N+d_i-1}$. The setting of Theorem \ref{DB_MIMO_rep} is an idealization that may not be satisfied in practice due to non-zero errors ${\scalebox{1.5}{$\epsilon$}}(\mathbf{u}_k,\tilde{\Xi}_k)$ and ${\delta}(\omega_k)$ in \eqref{noisy_basis}. Moreover, the persistency of excitation condition of $\{\hat{\Psi}_k(\mathbf{u},\Xi)\}_{k=0}^{N-1}$ can only be checked after collecting the input-output data but cannot be enforced a priori by a suitable design of $\mathbf{u}$. This is because the basis functions depend not only on the input but on the output as well. For the above reasons, we provide in the next section constructive methods to approximately solve the simulation and output-matching control problems in a data-based fashion and without requiring persistency of excitation. Furthermore, we provide qualitative error bounds on the difference between the estimated and true outputs and show that these errors tend to zero if \textbf{(i)} the noise in the data as well as the basis function approximation error tend to zero and \textbf{(ii)} persistency of excitation condition is satisfied.
	\section{Data-based Simulation \& Output-Matching}\label{sim_om_sec}
In this section, we investigate the data-based simulation and output-matching control problems for the class of DT-MIMO feedback linearizable nonlinear systems \eqref{NLsys}. For the nominal setting, where the basis function expansion in \eqref{basis} is exact and the output data is noiseless, the data-based simulation and output matching problems can be formulated in a similar manner as in \cite[Propositions 2 and 3]{AlsaltiBerLopAll2021}. For space reasons, we skip the nominal case here and consider the practically more relevant setting where $\scalebox{1.5}{$\epsilon$}(\mathbf{u}_k,\tilde{\Xi}_k)\not\equiv0, \delta(\omega_k)\not\equiv0$ in~\eqref{noisy_basis}. To do so, in the following we restrict our analysis to a compact subset of the input-state space $\Omega\subset\mathbb{R}^m\times\mathbb{R}^n$, i.e., we assume from here on that the a priori collected input and output trajectories as well as the simulated/matched trajectories evolve in the set $\Omega$. This, along with local Lipschitz continuity\footnote{According to \eqref{definition_of_v_and_Phi} and the discussion below it, $\Phi$ is the iterated composition of the continuously differentiable functions $\boldsymbol{f}$, $\boldsymbol{h}$ and $T^{-1}$ (which is continuously differentiable by \eqref{output_k} and \eqref{Xi}). Hence, $\Phi$ is locally Lipschitz continuous.} of $\Phi$ and the chosen basis functions, guarantees a uniform upper bound on the approximation error $\scalebox{1.5}{$\epsilon$}(\mathbf{u}_k,{{\Xi}_k})$ for all $(\mathbf{u}_k,{\Xi}_k)\in\Omega$. This assumption is summarized as follows.
	\begin{assumption}\label{bounded_err_assmp}
		The error in the basis function approximation $\hat{\scalebox{1.5}{$\epsilon$}}_k(\mathbf{u},{{\Xi}})$ is uniformly upper bounded by $\varepsilon^*>0$, i.e., $\norm*{\hat{\scalebox{1.5}{$\epsilon$}}_k(\mathbf{u},{{\Xi}})}_\infty\leq\varepsilon^*$, for all $(\mathbf{u}_k,{\Xi}_k)\in\Omega\subset\mathbb{R}^m\times\mathbb{R}^n$, where $\Omega$ is a compact subset of the input-state space.
\end{assumption}%
\noindent Since $\hat{E}_k(\mathbf{u},{\Xi})\coloneqq \mathcal{G}^\dagger\hat{\scalebox{1.5}{$\epsilon$}}_k(\mathbf{u},{\Xi})$, Assumption~\ref{bounded_err_assmp} implies
	\begin{equation}
		\hspace{0mm}
			\norm{\hat{E}_k(\mathbf{u},{\Xi})}_\infty\hspace{-1mm} \leq\norm{\mathcal{G}^\dagger}_\infty\norm{\hat{\scalebox{1.5}{$\epsilon$}}_k(\mathbf{u},{\Xi})}_\infty\hspace{-1mm}\leq\norm{\mathcal{G}^\dagger}_\infty\varepsilon^*\hspace{-1mm}.
\label{tilde_E_bdd_asmp}
	\end{equation}
\begin{remark}\label{Lipschitz_constants}
	\textup{\textbf{(a)}} We denote the Lipschitz constant of $\Phi$ w.r.t. $\Xi$ in the compact set $\Omega$ by $K_\Xi$. \textup{\textbf{(b)}} The function $\delta(\omega)$ in \eqref{noisy_basis} satisfies $\delta(\mathbf{0})=\mathbf{0}$ and, by local Lipschitz continuity of $\Phi$ and $\Psi$ w.r.t $\Xi$ on the compact set $\Omega$ and boundedness of $\mathbf{w}_k$, there exists a $K_w>0$ such that $\norm*{\delta(\omega_k)}_\infty\leq K_w w^*$ for all $k\geq0$.
\end{remark}
%%%%%%%%%%%%%%%%%%%%%%%%%%%%%%%%%%%%%%%%%%%%%%%%%%%%%%%
\subsection{Data-based simulation}\label{sim_sec}
The data-based simulation problem is defined as follows.
\begin{definition}\label{DD_sim_def}
	\textup{\textbf{Data-based simulation}\cite{Rapisarda08}\textbf{:}}  Given an input $\mathbf{\bar{u}}$ and initial conditions $\bar{\mathbf{x}}_0$ for the nonlinear system in \eqref{NLsys}, find the corresponding output trajectory $\mathbf{\bar{y}}$ using only input-output data, i.e., without explicitly identifying a model of the system.
\end{definition}
\indent In the following theorem, we solve a minimization problem for $\alpha$ in contrast to solving a set of nonlinear equations in \eqref{inexact_result}. Once a solution is obtained, an approximate output trajectory is found and its difference from the true simulated output is shown to be bounded.\par
In what follows, we use $\hat{\Psi}(\mathbf{u},\tilde{\Xi})$ to denote the stacked vector of $\{\hat{\Psi}_k(\mathbf{u},\tilde{\Xi})\}_{k=0}^{N-1}$. Moreover, for some input to be simulated $\bar{\mathbf{u}}$ and some vector $\alpha\in\mathbb{R}^{N-L+1}$, we use $\hat{\Psi}(\mathbf{\bar{u}},H_{L+1}(\tilde{\Xi})\alpha)$ to denote the stacked vector of the sequence $\{\hat{\Psi}_k(\mathbf{\bar{u}},H_{L+1}(\tilde{\Xi})\alpha)\}_{k=0}^{L-1}$, with each element defined as $\hat{\Psi}_k(\mathbf{\bar{u}},H_{L+1}(\tilde{\Xi})\alpha)\coloneqq\Psi(\mathbf{\bar{u}}_k,H_1(\tilde{\Xi}_{[k,k+N-L]})\alpha)$.
\begin{theorem}\label{MIMO_inexact_sim_thm}
		Suppose Assumptions \ref{MIMO_img_assmp}-\ref{bounded_err_assmp} are satisfied and let $\{\mathbf{u}_k\}_{k=0}^{N-1}$, $\{\tilde{y}_{i,k}\}_{k=0}^{N+d_i-1}$, for $i\in\mathbb{Z}_{[1,m]}$, be input-output data sequences collected from \eqref{NLsys}. Furthermore, let \(\{\mathbf{\bar{u}}_k\}_{k=0}^{L-1}\) be a new input to be simulated with \(\bar{\Xi}_0=\begin{bmatrix}
			\bar{y}_{1,[0,d_1-1]}^\top \hspace{-0.5mm}& \hspace{-0.5mm}\dots \hspace{-0.5mm}& \hspace{-0.5mm}\bar{y}_{m,[0,d_m-1]}^\top
		\end{bmatrix}^\top\) specifying initial conditions for the state $\bar{\Xi}$ in \eqref{Lsys3}. Let the following optimization problem be feasible for the given $\bar{\Xi}_0$
		\begin{subequations}
			\begin{align}
				\hspace{-2mm}\alpha^{*}\in&\argmin\limits_{\alpha} J(\alpha)\hspace{-0.5mm}\coloneqq\hspace{-0.5mm} \norm*{\mathcal{H}}_2^2 \hspace{-0.5mm}+\hspace{-0.5mm} \lambda\max\{\hspace{-0.25mm}\varepsilon^*\hspace{-0.75mm},w^*\hspace{-0.5mm}\}\norm*{\alpha}_2^2,\label{alpha*_a}\\
				&\textup{s.t. } \bar{\Xi}_0 = H_1(\tilde{\Xi}_{[0,N-L]})\alpha,\label{alpha*_b}
			\end{align}\label{alpha*}%
		\end{subequations}%
		where $\lambda>0$, $\mathcal{H}\hspace{-0.5mm}\coloneqq
		H_L(\hat{\Psi}(\mathbf{u},\tilde{\Xi}))
		\alpha - 
		\hat{\Psi}(\mathbf{\bar{u}},H_{L+1}(\tilde{\Xi})\alpha)$. Then, the estimated simulated outputs are given by $\hat{y}_{i,[0,L+d_i-1]} = H_{L+d_i}(\tilde{y}_i)\alpha^*$, for $i\in\mathbb{Z}_{[1,m]}$, and the error ${e}_{i}\vcentcolon=\bar{y}_{i}-{\hat{y}}_{i}$ satisfies $e_{i,[0,d_i-1]}=\mathbf{0}$ and is upper bounded by
		\begin{align}
			|{e}_{i,k+d_i}|&\leq \mathcal{P}^{k}(K_\Xi)\big( \varepsilon^*(1+\norm*{\alpha^*}_1)  + \norm*{\mathcal{G}}_{\infty}\sqrt{b}\notag\\
			&\qquad\qquad\quad+w^*(1+K_w)\norm*{\alpha^*}_1\big),\label{err_sim}
		\end{align}
		for all $k\in\mathbb{Z}_{[0,L-1]}$, where $K_\Xi$ and $K_w$ are defined in Remark~\ref{Lipschitz_constants}, $b=J(\alpha^*)-\lambda\max\{\varepsilon^*,w^*\}\norm*{\alpha^*}_2^2$ and $\mathcal{P}^k(K_\Xi)=(K_\Xi)^k + (K_\Xi)^{k-1} + \dots + K_\Xi + 1$.
\end{theorem}
\begin{proof}
		Let $\hat{\Xi}\coloneqq H_{L+1}(\tilde{\Xi})\alpha^*$. By definition of $\alpha^*$ from \eqref{alpha*} and for some vector $c$ satisfying $c^\top c=b$, it holds that
		\begin{equation}
			H_L(\hat{\Psi}(\mathbf{u},\tilde{\Xi}))\alpha^* = \hat{\Psi}(\mathbf{\bar{u}},\hat{\Xi})+ c.\label{sim_opt}
		\end{equation}
		The constraint in \eqref{alpha*_b} fixes the initial estimated state $\hat{\Xi}_0 \coloneqq H_1(\tilde{\Xi}_{[0,N-L]})\alpha^* = \bar{\Xi}_0$ and hence, $e_{i,[0,d_i-1]}=\bar{y}_{i,[0,d_i-1]}-\hat{y}_{i,[0,d_i-1]} = \mathbf{0}$ (cf. \eqref{Xi}). Furthermore, each estimated simulated output takes the form $\hat{y}_{i,k+d_i}=H_1(\tilde{y}_{i,[k+d_i,k+d_i+N-L]})\alpha^*$, for all $k\in\mathbb{Z}_{[0,L-1]}$. Therefore, one can write
		\begin{align}
			&\hat{y}_{i,k+d_i}\label{proof_sim}\\
			&=H_1({y}_{i,[k+d_i,k+d_i+N-L]})\alpha^* + H_1(w_{i,[k+d_i,k+d_i+N-L]})\alpha^*\notag
			\\&\stackrel{\eqref{y_for_proofs}}{=}g_i^\top H_1\Big(\hat{\Psi}_{[k,k+N-L]}(\mathbf{u},\tilde{\Xi})
			+\hat{E}_{[k,k+N-L]}(\mathbf{u},\tilde{\Xi})\notag\\
			&\quad+ \hat{D}_{[k,k+N-L]}(\omega)\Big)\alpha^*+ H_1(w_{i,[k+d_i,k+d_i+N-L]})\alpha^*\notag\\
			&\stackrel{\eqref{imp_defs}}{=} \hspace{-0.5mm}g_i^\top\hspace{-0.5mm} H_1\hspace{-0.5mm}(\hat{\Psi}_{[k,k+N-L]}(\mathbf{u},\tilde{\Xi}))\hspace{-0.25mm}\alpha^*\hspace{-0.5mm}+\hspace{-0.5mm} H_1\hspace{-0.5mm}(\hat{\varepsilon}_{i,[k,k+N-L]}(\mathbf{u},\tilde{\Xi}))\hspace{-0.25mm}\alpha^*\notag
			\\&\quad+ H_1(\hat{\delta}_{i,[k,k+N-L]}(\omega))\alpha^* + H_1(w_{i,[k+d_i,k+d_i+N-L]})\alpha^*\notag
			\\&\stackrel{\eqref{sim_opt}}{=} g_i^\top{\Psi}(\mathbf{\bar{u}}_{k},\hat{\Xi}_{k}) \hspace{-0.5mm}+ \hspace{-0.5mm}g_i^\top c_{k}\hspace{-0.5mm}+\hspace{-0.5mm} H_1\hspace{-0.5mm}(\hat{\varepsilon}_{i,[k,k+N-L]}(\mathbf{u},\tilde{\Xi}))\alpha^*\notag
			\\&\quad+ H_1(\hat{\delta}_{i,[k,k+N-L]}(\omega))\alpha^*+ H_1(w_{i,[k+d_i,k+d_i+N-L]})\alpha^*\notag
			\\&\stackrel{\eqref{basis}}{=}\phi_i(\mathbf{\bar{u}}_{k},\hat{\Xi}_{k}) - \varepsilon_i(\mathbf{\bar{u}}_{k},\hat{\Xi}_{k}) + g_i^\top c_{k}\notag%
			\\&\quad+\hspace{-1mm}H_1\hspace{-0.5mm}(\hat{\varepsilon}_{i,[k,k+N-L]}(\mathbf{u},\tilde{\Xi}))\hspace{-0.25mm}\alpha^*\hspace{-1mm}+\hspace{-1mm} H_1(w_{i,[k+d_i,k+d_i+N-L]})\alpha^*\hspace{-1mm}\notag\\
			&\quad+ H_1(\hat{\delta}_{i,[k,k+N-L]}(\omega))\alpha^*\notag,
		\end{align}
		where $c_{k}$ is the $k-$th entry of the vector $c$. The true, unknown output $\bar{y}_{i,k+d_i}$, for $i\in\mathbb{Z}_{[1,m]}$ and $k\in\mathbb{Z}_{[0,L-1]}$, can be written as $\bar{{y}}_{i,k+d_i} \stackrel{\eqref{y_for_proofs}}{=}\phi_i(\mathbf{\bar{u}}_{k},\bar{\Xi}_{k})$. Therefore, the error is expressed as
		\begin{align}
			e_{i,k+d_i} &= \phi_i(\mathbf{\bar{u}}_{k},\bar{\Xi}_{k}) - \phi_i(\mathbf{\bar{u}}_{k},\hat{\Xi}_{k}) + \varepsilon_i(\mathbf{\bar{u}}_{k},\hat{\Xi}_{k}) - g_i^\top c_{k}\notag\\
			&\, - H_1\hspace{-0.5mm}(\hat{\varepsilon}_{i,[k,k+N-L]}(\mathbf{u},\tilde{\Xi}))\hspace{-0.25mm}\alpha^*\hspace{-1mm}-\hspace{-1mm}H_1(\hat{\delta}_{i,[k,k+N-L]}(\omega))\alpha^*\notag\\
			&\, -H_1(w_{i,[k+d_i,k+d_i+N-L]})\alpha^*.\label{sim_e_vector}
		\end{align}
		The expression in \eqref{sim_e_vector} can be upper bounded by
		\begin{align}
			|{{e}_{i,k+d_i}}|\hspace{-0.25mm} &\leq\hspace{-0.25mm} K_\Xi \norm*{\bar{\Xi}_{k} \hspace{-0.25mm}-\hspace{-0.25mm} \hat{\Xi}_{k}}_\infty \hspace{-2.5mm}+\hspace{-0.25mm} \varepsilon^*(1\hspace{-0.25mm}+\hspace{-0.25mm}\norm*{\alpha^*}_1\hspace{-0.25mm}) \hspace{-0.25mm} + \hspace{-0.25mm}\norm*{\mathcal{G}}_{\infty}\hspace{-0.25mm}\sqrt{b}\notag
			\\&\quad+w^*(1+K_w)\norm*{\alpha^*}_1,\label{sim_bnd_b4_recursion}
		\end{align}
		where the first two terms in \eqref{sim_e_vector} were bounded by Lipschitz continuity of $\phi_i$ and the third and fifth terms were bounded by $\varepsilon^*(1+\norm*{\alpha^*}_1)$ following Assumption \ref{bounded_err_assmp}. The fourth term was bounded by $\norm*{g_i^\top}_\infty\norm{c_{k}}_{\infty}\leq\norm*{\mathcal{G}}_{\infty}\norm*{c}_\infty\leq\norm*{\mathcal{G}}_{\infty}\sqrt{c^\top c} = \norm*{\mathcal{G}}_{\infty}\sqrt{b}$, and the last two terms were bounded by $w^*(1+K_w)\norm*{\alpha^*}_1$ since $\norm*{\mathbf{w}}_\infty\leq w^*$ by assumption. We continue the proof by induction. Let $k=0$ in \eqref{sim_bnd_b4_recursion}, and notice that
	\begin{align}
		\hspace{-1mm}|{{e}_{i,d_i}}|\hspace{-0.5mm}&\leq\hspace{-0.5mm} \varepsilon^*\hspace{-0.25mm}(1\hspace{-0.5mm}+\hspace{-0.5mm}\norm*{\alpha^*}_1\hspace{-0.25mm}) \hspace{-0.5mm} + \hspace{-0.5mm}\norm*{\mathcal{G}}_{\hspace{-0.25mm}\infty}\hspace{-0.5mm}\sqrt{b} +w^*\hspace{-0.25mm}(1\hspace{-0.5mm}+\hspace{-0.5mm}K_w)\norm*{\alpha^*}_1\hspace{-0.25mm}\label{induction_base},
	\end{align}
		since $\bar{\Xi}_0 = \hat{\Xi}_0$ from \eqref{alpha*_b}. Notice that \eqref{induction_base} has the form \eqref{err_sim} with $\mathcal{P}^0(K_\Xi)=1$. For the induction step, let the following hold for all $k\in\mathbb{Z}_{[1,L-1]}$, all $\bar{k}\in\mathbb{Z}_{[0,k-1]}$ and $i\in\mathbb{Z}_{[1,m]}$
		\begin{align}
			|{e}_{i,\bar{k}+d_i}|&\leq \mathcal{P}^{\bar{k}}(K_\Xi)( \varepsilon^*(1+\norm*{\alpha^*}_1)  + \norm*{\mathcal{G}}_{\infty}\sqrt{b}\notag\\
			&\qquad\qquad\quad+w^*(1+K_w)\norm*{\alpha^*}_1),\label{induction_step1}
		\end{align}
		Since $\mathcal{P}^{\bar{k}}(K_\Xi)$ increases with increasing $\bar{k}$, this implies that the following bound on the previous error instances holds
		\begin{align}
			\hspace{-2mm}\norm*{e_{i,[0,k+d_i-1]}}_\infty \hspace{-0.5mm} &\leq \hspace{-0.75mm} \mathcal{P}^{(\hspace{-0.1mm}k\hspace{-0.1mm}-\hspace{-0.1mm}1\hspace{-0.1mm})}\hspace{-0.5mm}(\hspace{-0.15mm}K_\Xi\hspace{-0.15mm})\hspace{-0.5mm}(\hspace{-0.25mm} \varepsilon^*(1\hspace{-0.25mm}+\hspace{-0.25mm}\norm*{\alpha^*}_1\hspace{-0.25mm}) \hspace{-0.25mm} + \hspace{-0.25mm}\norm*{\mathcal{G}}_{\infty}\hspace{-0.25mm}\sqrt{b}\notag\\
			&\quad+w^*(1+K_w)\norm*{\alpha^*}_1).\label{induction_allk}
		\end{align}
		Then, by \eqref{sim_bnd_b4_recursion}, the definition of $\bar{\Xi}_k$ as in \eqref{Xi} and the corresponding $\hat{\Xi}_{k}\hspace{-1.25mm}=\hspace{-1.25mm}\begin{bmatrix}\hat{y}_{1,[k,k+d_1-1]}^\top\hspace{-1mm}&\hspace{-1mm}\cdots\hspace{-1mm}&\hspace{-1mm}\hat{y}_{m,[k,k+d_m-1]}^\top\end{bmatrix}^\top\hspace{-2mm}$, we have
		\begin{align}
			\label{induction_step2}
			|{{e}_{i,k+d_i}}| \hspace{-0.5mm}&\leq\hspace{-0.5mm} K_\Xi\norm*{\begin{bmatrix}
					e_{1,[k,k+d_1-1]} \\ \vdots \\ e_{m,[k,k+d_m-1]}
			\end{bmatrix}}_\infty\hspace{-3mm} +\varepsilon^*(1\hspace{-0.25mm}+\hspace{-0.25mm}\norm*{\alpha^*}_1\hspace{-0.25mm}) \hspace{-0.25mm}\\
			&\qquad\hspace{-0.25mm}+w^*(1+K_w)\norm*{\alpha^*}_1 + \hspace{-0.25mm}\norm*{\mathcal{G}}_{\infty}\hspace{-0.25mm}\sqrt{b}.\notag
		\end{align}
		Notice that the elements of the first term on the RHS of \eqref{induction_step2} are bounded by \eqref{induction_allk}. Therefore, we have that
		\begin{align*}
			|{{e}_{i,k+d_i}}| \hspace{-0.5mm}&\leq\hspace{-0.5mm} K_\Xi \left(\mathcal{P}^{(\hspace{-0.1mm}k\hspace{-0.1mm}-\hspace{-0.1mm}1\hspace{-0.1mm})}\hspace{-0.5mm}(\hspace{-0.15mm}K_\Xi\hspace{-0.15mm})(\varepsilon^*(1\hspace{-0.25mm}+\hspace{-0.25mm}\norm*{\alpha^*}_1\hspace{-0.25mm}) + \hspace{-0.25mm}\norm*{\mathcal{G}}_{\infty}\hspace{-0.25mm}\sqrt{b}\right.\\
			&\,+w^*(1+K_w)\norm*{\alpha^*}_1)\Big) +\varepsilon^*(1\hspace{-0.25mm}+\hspace{-0.25mm}\norm*{\alpha^*}_1\hspace{-0.25mm})\\
			&+ \hspace{-0.25mm}\norm*{\mathcal{G}}_{\infty}\hspace{-0.25mm}\sqrt{b}\,+w^*(1+K_w)\norm*{\alpha^*}_1.
		\end{align*}
		Collecting the terms in the last inequality results in \eqref{err_sim}, which completes the proof.
\end{proof}
{Theorem \ref{MIMO_inexact_sim_thm} provides an approximate solution to the data-based simulation problem when the data is noisy and the basis functions approximation error is unknown but uniformly upper bound as in Assumption \ref{bounded_err_assmp}. This was done by solving the optimization problem in \eqref{alpha*}. In particular, if \eqref{alpha*} is feasible (i.e., the constraint \eqref{alpha*_b} can be satisfied) then an approximate solution can be found for the simulation problem. Notice that persistency of excitation condition is \emph{not} necessary for feasibility of \eqref{alpha*_b} and hence is not necessarily required for Theorem \ref{MIMO_inexact_sim_thm}. In fact, feasibility of \eqref{alpha*} is given for any $\bar{\Xi}_0$ if $H_1(\tilde{\Xi}_{[0,N-L]})$ has full row rank, which can be easily verified given the a priori collected data $\{\tilde{\Xi}_k\}_{k=0}^{N}$ (which is given by $\{\tilde{y}_{i,k}\}_{k=0}^{N+d_i-1}$ for $i\in\mathbb{Z}_{[1,m]}$). Furthermore, Theorem \ref{MIMO_inexact_sim_thm} provides a bound on how far the estimated outputs $\hat{y}_{i,[0,L+d_i-1]}$ deviate from the actual system response $\bar{y}_{i,[0,L+d_i-1]}$.}\par
Since the solution of \eqref{alpha*} appears in the error bound, we incentivize solutions with smaller norm using a regularization term as in \eqref{alpha*} which depends on the basis functions approximation error and noise bounds $\varepsilon^*$ and $w^*$, respectively. It can be seen from \eqref{err_sim} that the error bound also depends on the unknown matrix of coefficients $\mathcal{G}$ in \eqref{basis}. Under certain assumptions, one can compute an upper bound on its norm in a model-free fashion (see \cite[Lemma 2]{Alsalti2021c}).\par
In general, the error bound \eqref{err_sim} can be conservative due to the use of Lipschitz arguments. In particular, the polynomial $\mathcal{P}^k(K_{\Xi})$ diverges (as $k\to\infty$) for $K_{\Xi}>1$ but converges to $\frac{1}{1-K_{\Xi}}$ for $K_{\Xi}<1$, in which case the error bound in \eqref{err_sim} is uniformly upper bounded for all $k$. In contrast, the actual error $e_{i,k}$ can potentially be (much) smaller than \eqref{err_sim}, depending on the quality of the collected data and the chosen basis functions.\par
The following lemma establishes the important qualitative result that if $\{\hat{\Psi}_k(\mathbf{u},\tilde{\Xi})\}_{k=0}^{N-1}$ is persistently exciting, then the error bound in \eqref{err_sim} (and hence the true error) goes to zero as $\max\{\varepsilon^*,w^*\}$ goes to zero. 
\begin{lemma}\label{lemma_eps_zero}
Let the conditions of Theorem \ref{MIMO_inexact_sim_thm} hold and let $\{\hat{\Psi}_k(\mathbf{u},\tilde{\Xi})\}_{k=0}^{N-1}$ be persistently exciting of order $L+n$. Then, as $\nu\coloneqq\max\{\varepsilon^*,w^*\}\to0$, the error $e_{i,[0,L+d_i-1]}\to0$.
\end{lemma}
\begin{proof}
	It can be seen from \eqref{err_sim} that, for each fixed $k$, the error bound goes to zero as $\nu\to 0$ \emph{if} \textbf{(i)} $\nu\norm*{\alpha^*}_1\to0$ as $\nu\to0$ and \textbf{(ii)} $b\to0$ as $\nu\to 0$. To show these two properties, notice that for sufficiently small $\varepsilon^*$ and $w^*$, persistency of excitation of $\{\hat{\Psi}_k(\mathbf{u},\tilde{\Xi})\}_{k=0}^{N-1}$ implies that of $\{\hat{\Psi}_k(\mathbf{u},\tilde{\Xi}) + \hat{E}_k(\mathbf{u},\tilde{\Xi}) + \hat{D}_k(\omega)\}_{k=0}^{N-1}$ of the same order. This is the case, e.g., when $\norm{H_L(\hat{E}(\mathbf{u},\tilde{\Xi}) + \hat{D}(\omega))}_2<\sigma_{\textup{min}}(H_L(\hat{\Psi}(\mathbf{u},\tilde{\Xi})))$. Now, we propose the following candidate solution to \eqref{alpha*}
	\begin{gather}
			\hspace{-2mm}\bar{\alpha}\hspace{-1mm} = \hspace{-1mm}\begin{bmatrix}\hspace{-0.75mm}
				H_{L}(\hat{\Psi}(\mathbf{u},\tilde{\Xi}) \hspace{-1mm}+\hspace{-1mm} \hat{E}(\mathbf{u},\tilde{\Xi}) \hspace{-1mm}+ \hspace{-1mm}\hat{D}(\omega))\hspace{-0.75mm}\\ 	H_{1}(\tilde{\Xi}_{[0,N-L]})\end{bmatrix}^{\hspace{-0.5mm}\dagger}\hspace{-0.75mm}\begin{bmatrix}
				\hspace{-0.5mm}\hat{\Psi}(\mathbf{\bar{u}},\bar{\Xi}) \hspace{-1mm}+\hspace{-1mm} \hat{E}(\mathbf{\bar{u}},\bar{\Xi})\hspace{-0.5mm}\\ 	\bar{\Xi}_0
			\end{bmatrix}\hspace{-1.5mm}.\label{alpha_bar_definition}
	\end{gather}
\begin{figure*}[!t]
	\normalsize
	\begin{subequations}
			\begin{align}
				\overbrace{\norm*{\hspace{-0.25mm}H_L(\hat{\Psi}(\mathbf{u},\tilde{\Xi}))\alpha^*\hspace{-0.5mm} -\hspace{-0.5mm}\hat{\Psi}(\mathbf{\bar{u}},H_{L+1}(\tilde{\Xi})\alpha^*)\hspace{-0.25mm}}_2^2}^{=b}\hspace{-0.5mm}+\lambda\nu\norm*{\alpha^*}_2^2
				&\leq
				\norm*{\hspace{-0.25mm}H_{L}(\hat{\Psi}(\mathbf{u},\tilde{\Xi}))\bar{\alpha}\hspace{-0.5mm} -\hspace{-0.5mm}\hat{\Psi}(\mathbf{\bar{u}},H_{L+1}(\tilde{\Xi})\bar{\alpha}) \hspace{-0.25mm}}_2^2\hspace{-0.5mm}+ \lambda\nu\norm*{\bar{\alpha}}_2^2.
				\label{RHS_to_0a}
				\\
				b+\lambda\nu\norm*{\alpha^*}_2^2
				&\leq{\norm*{\Pi(w^*,\varepsilon^*)}_2^2}+ \lambda\nu\norm*{\bar{\alpha}}_2^2.\label{RHS_to_0b}
			\end{align}\label{RHS_to_0}
	\end{subequations}%
	\hrulefill
\end{figure*}%
	Notice that \eqref{alpha_bar_definition} is such that \eqref{alpha*_b} is satisfied since the pseudoinverse is in fact a right inverse given the implied persistency of excitation condition\footnote{According to \cite{Willems05}, such a matrix with the lower block row of the form $H_1(\Xi_{[0,N-L]})$ has full row rank. For sufficiently small $w^*$, the considered matrix still maintains full row rank and, hence, a right inverse exists.\label{imp_footnote}} (cf. \cite{Willems05}). By optimality, we have that $J(\alpha^*)\leq J(\bar{\alpha})$ (expanded in \eqref{RHS_to_0}), where $\Pi(w^*,\varepsilon^*)\coloneqq \hat{E}(\mathbf{\bar{u}},\bar{\Xi})-H_L(\hat{E}(\mathbf{{u}},\tilde{\Xi})+\hat{D}(\omega))\bar{\alpha} -  \hat{\Psi}(\mathbf{\bar{u}},H_{L+1}(\tilde{\Xi})\bar{\alpha}) + \hat{\Psi}(\mathbf{\bar{u}},\bar{\Xi})$ was obtained by plugging \eqref{alpha_bar_definition} into \eqref{RHS_to_0a}.\par
	If the right hand side of \eqref{RHS_to_0} goes to zero as $\nu\to0$, then the left-hand side of \eqref{RHS_to_0}, which is a sum of two non-negative terms, goes to zero as well, implying both \textbf{(i)} and \textbf{(ii)} above. Clearly, the term $\lambda\nu\norm{\bar{\alpha}}_2$ on the right-hand side of \eqref{RHS_to_0} goes to zero as $\nu\to0$ as long as $\norm{\bar{\alpha}}_2$ is bounded. This is also the case for the first two terms of the definition of $\Pi(\varepsilon^*,w^*)$. For the last two terms of $\Pi(\varepsilon^*,w^*)$ to cancel out, we need to show that $\bar{\alpha}\to\tilde{\alpha}$ as $\nu\to0$, where $\tilde{\alpha}$ is the {bounded} solution of the simulation problem in the nominal (unperturbed) setting (see \eqref{inexact_result}). This is done in the following.\par
	Recall from \eqref{alpha_bar_definition} that $\bar{\alpha}$ satisfies the following
	\begin{gather*}
\begin{bmatrix}\hspace{-0.5mm}
	H_{L}(\hat{\Psi}(\mathbf{u},\tilde{\Xi}) \hspace{-1mm}+\hspace{-1mm} \hat{E}(\mathbf{u},\tilde{\Xi}) \hspace{-1mm}+ \hspace{-1mm}\hat{D}(\omega))\hspace{-0.5mm}\\ 	H_{1}(\tilde{\Xi}_{[0,N-L]})\end{bmatrix}\hspace{-1mm}\bar{\alpha} = \hspace{-0.75mm}\begin{bmatrix}
	\hspace{-0.25mm}\hat{\Psi}(\mathbf{\bar{u}},\bar{\Xi}) \hspace{-1mm}+\hspace{-1mm} \hat{E}(\mathbf{\bar{u}},\bar{\Xi})\hspace{-0.25mm}\\ 	\bar{\Xi}_0
\end{bmatrix}\hspace{-1mm},
	\end{gather*}%
	which, for $\Delta\hat{\Psi}\coloneqq\hat{\Psi}(\mathbf{u},\tilde{\Xi})-\hat{\Psi}(\mathbf{u},\Xi)$, can be rewritten as
	\begin{gather}
		\hspace{-5mm}\Bigg(\overbrace{\begin{bmatrix}H_{L}(\hat{\Psi}(\mathbf{u},{\Xi}))\\ H_{1}({\Xi}_{[0,N-L]})\end{bmatrix}}^{\eqqcolon \Theta} \hspace{-1mm}+\hspace{-1mm} \overbrace{\begin{bmatrix}H_{L}(\Delta\hat{\Psi}+\hat{E}(\mathbf{u},\tilde{\Xi}) + \hat{D}(\omega))\\ H_1(\omega_{[0,N-L]})\end{bmatrix}}^{\eqqcolon \Delta \Theta}\Bigg) \bar{\alpha}\notag\\
		\hspace{15mm}=\Bigg(\overbrace{\begin{bmatrix}\hat{\Psi}(\mathbf{\bar{u}},\bar{\Xi})\\ \bar{\Xi}_0 \end{bmatrix}}^{\eqqcolon \mu} \hspace{-0.5mm}+\hspace{-0.5mm} \overbrace{\begin{bmatrix}\hat{E}(\mathbf{\bar{u}},\bar{\Xi}) \\ \boldsymbol{0}\end{bmatrix}}^{\eqqcolon \Delta \mu}\Bigg).\label{Golub}
	\end{gather}
	For a \emph{fixed} simulated trajectory $\mathbf{\bar{u}},\bar{\Xi}$, \eqref{Golub} represents a perturbed under-determined system of linear equations. Therefore, by applying results from perturbation theory of under-determined systems \cite[Theorem 5.6.1]{Golub12}, it can be shown that
	\begin{equation}
		\frac{\norm*{\bar{\alpha}\hspace{-0.5mm}-\hspace{-0.5mm}\tilde{\alpha}}_2}{\norm*{\tilde{\alpha}}_2}\hspace{-0.5mm}\leq\hspace{-0.5mm}\kappa_2(\Theta)\hspace{-0.5mm}\left(\hspace{-0.5mm}c_0\frac{\norm*{\Delta \Theta}_2}{\norm*{\Theta}_2} \hspace{-0.5mm}+\hspace{-0.5mm} \frac{\norm*{\Delta \mu}_2}{\norm*{\mu}_2}\hspace{-0.5mm}\right)\hspace{-0.5mm} +\hspace{-0.5mm} O(\varphi^2),\label{pertb_lin_sys}
	\end{equation}
	where $\tilde{\alpha}=\Theta^\dagger \mu$ is the \textit{bounded} minimum-norm solution of the unperturbed system of equations\footnote{For sufficiently small $\varepsilon^*,w^*$, PE of $\{\hat{\Psi}_k(\mathbf{u},\tilde{\Xi})\}_{k=0}^{N-1}$ implies that $\Theta+\Delta\Theta$ has full row rank (see footnote \ref{imp_footnote}). Similarly, it holds that $\Theta$ also has full row rank and, hence, the unperturbed set of equations has a solution.} and $\kappa_2(\Theta)=\norm*{\Theta}_2\norm*{\Theta^\dagger}_2$ is the condition number of the matrix $\Theta$. Moreover,\footnote{$n_{c,\Theta},\,n_{r,\Theta}$ are the numbers of columns and rows of $\Theta$, respectively.} $c_0\coloneqq\min{\left\lbrace2, n_{c,\Theta}-n_{r,\Theta}+1\right\rbrace}$ and $O(\varphi^2)$ are higher order terms with $\varphi\coloneqq\max{\left\lbrace\frac{\norm*{\Delta \Theta}_2}{\norm*{\Theta}_2},\frac{\norm*{\Delta \mu}_2}{\norm*{\mu}_2}\right\rbrace}$.\par
	As $\nu\to0$, the terms $\norm{\Delta\Theta},\,\norm{\Delta\mu}\to0$ due to Assumption~\ref{bounded_err_assmp} and local Lipschitz continuity of $\Psi$ (see \eqref{Golub}). Hence, it can be seen from \eqref{pertb_lin_sys} that the candidate $\bar{\alpha}$ approaches the bounded solution $\tilde{\alpha}$, which is finite for any fixed $\mathbf{\bar{u}},\bar{\Xi}$.
\end{proof}
Lemma \ref{lemma_eps_zero} shows that the error bound in \eqref{err_sim} (and, hence, the true error) goes to zero as $\max\{\varepsilon^*,w^*\}$ goes to zero if persistency of excitation of $\{\hat{\Psi}_k(\mathbf{u},\tilde{\Xi})\}_{k=0}^{N-1}$ is satisfied. In the next subsection, we study the data-based output matching problem for full-state feedback linearizable systems.
%%%%%%%%%%%%%%%%%%%%%%%%%%%%%%%%%%%%%%%%%%%%%%%%%%%%%%%%%%%%%%%%%%%%%%%%%%%%%%%%%%%%%%%
\subsection{Data-based output-matching}\label{om_sec}
The data-based output matching control problem is defined as follows.
\begin{definition}\label{DD_om_def}
	\textup{\textbf{Data-based output-matching control}\cite{Rapisarda08}\textbf{:}} Given a desired reference trajectory $\mathbf{\bar{y}}$ and a corresponding initial condition $\bar{\mathbf{x}}_0$ for the nonlinear system in \eqref{NLsys}, find the required input trajectory $\mathbf{\bar{u}}$ that, when applied to the system, results in an output that tracks the desired reference trajectory, using only input-output data.
\end{definition}
Analogous to the discussion in Section \ref{sim_sec}, we propose a similar approach to the one shown in Theorem \ref{MIMO_inexact_sim_thm} for solving the data-based output-matching control problem despite an unknown basis function approximation error and using only input and noisy output data. In particular, we assume that the a priori collected data and the trajectory to be matched evolve in a compact subset of the input-state space as in Assumption~\ref{bounded_err_assmp}. Next, we solve for an approximate control input $\hat{\mathbf{u}}_{[0,L-1]}$ that, when applied to the system, results in approximate output trajectories $\hat{y}_{i,[0,L+d_i-1]}$ which track the reference trajectories $\bar{y}_{i,[0,L+d_i-1]}$ as closely as possible. The error between the two outputs is bounded as shown in  Theorem \ref{MIMO_inexact_om_thm}. Before presenting the result, the following assumption is made which is needed in order to retrieve the desired input $\mathbf{\bar{u}}$.
\begin{assumption}\label{u_in_span_asmp}
		Let $\psi_j(\mathbf{u}_k,\Xi_k) = {u}_{j,k}$ for all $j\in\mathbb{Z}_{[1,m]}$.
\end{assumption}
Assumption \ref{u_in_span_asmp} is not restrictive since it can be replaced by only requiring $\mathbf{u}_k$ to lie in the span of ${\Psi}$, or that for any $\Xi$, $\Psi(\cdot,\Xi)$ is injective.\par
The following theorem is the dual result of Theorem \ref{MIMO_inexact_sim_thm}. For some vector $\alpha\in\mathbb{R}^{N-L+1}$ and given reference trajectories $\bar{y}_{i}$ (with $\bar{\Xi}$ being the corresponding transformed state (see \eqref{Xi})), we use $\hat{\Psi}(H_{L}(\mathbf{u})\alpha,\bar{\Xi})$ to denote the stacked vector of the sequence $\{\hat{\Psi}_k(H_{L}(\mathbf{u})\alpha,\bar{\Xi})\}_{k=0}^{L-1}$ with each element defined as $\hat{\Psi}_k(H_{L}(\mathbf{u})\alpha,\bar{\Xi})\coloneqq \Psi(H_1(\mathbf{u}_{[k,k+N-L]})\alpha,\bar{\Xi}_k)$.
\begin{theorem}\label{MIMO_inexact_om_thm}
		Suppose Assumptions \ref{MIMO_img_assmp}--\ref{u_in_span_asmp} are satisfied and let $\{\mathbf{u}_k\}_{k=0}^{N-1}$, $\{\tilde{y}_{i,k}\}_{k=0}^{N+d_i-1}$, for $i\in\mathbb{Z}_{[1,m]}$, be input-output data sequences collected from \eqref{NLsys}. Furthermore, let \(\{\bar{y}_{i,k}\}_{k=0}^{L+d_i-1}\) be desired reference trajectories with \(\bar{\Xi}_0=\begin{bmatrix}
			\bar{y}_{1,[0,d_1-1]}^\top & \dots & \bar{y}_{m,[0,d_m-1]}^\top
		\end{bmatrix}^\top\) specifying the initial condition for the state $\bar{\Xi}$ in \eqref{Lsys3}. Let the following optimization problem be feasible for the given $\bar{\Xi}_0$
		\begin{subequations}
			\begin{align}
				\hspace{-2mm}\alpha^{*}\in&\argmin\limits_{\alpha} J(\alpha)\hspace{-0.5mm}\coloneqq\hspace{-0.5mm} \norm*{\mathcal{H}}_2^2 \hspace{-0.5mm}+\hspace{-0.5mm} \lambda\max\{\hspace{-0.25mm}\varepsilon^*\hspace{-0.75mm},w^*\hspace{-0.5mm}\}\norm*{\alpha}_2^2,\label{alpha*2_a}\\
				&\textup{s.t. } \bar{\Xi}_0 = H_1(\tilde{\Xi}_{[0,N-L]})\alpha,\label{alpha*2_b}
			\end{align}\label{alpha*2}%
		\end{subequations}%
		with $\lambda>0$, $\mathcal{H}\coloneqq\begin{bmatrix}
			H_{L}(\hat{\Psi}(\mathbf{u},\tilde{\Xi}))\\ H_{L+1}(\tilde{\Xi})
		\end{bmatrix}\alpha - \begin{bmatrix}
			\hat{\Psi}(H_{L}(\mathbf{u})\alpha,\bar{\Xi})\\ \bar{\Xi}
		\end{bmatrix}$. Then, $\hat{\mathbf{u}}\coloneqq H_{L}(\mathbf{u})\alpha^*$ is the estimated input to the system that, when applied to \eqref{NLsys}, results in the approximate outputs $\hat{y}_{i,[0,L+d_i-1]}$, for $i\in\mathbb{Z}_{[1,m]}$. Furthermore, the error ${e}_{i}\vcentcolon=\bar{y}_{i}-\hat{y}_{i}$ satisfies $e_{i,[0,d_i-1]}=\mathbf{0}$ and is upper bounded by
		\begin{align}
			\hspace{-1mm}|{e}_{i,k+d_i}\hspace{-0.25mm}|\hspace{-0.25mm}&\leq\hspace{-0.25mm} \mathcal{P}^{k}(\hspace{-0.25mm}K_\Xi\hspace{-0.25mm})\big(\hspace{-0.25mm} \varepsilon^*(1\hspace{-0.25mm}+\hspace{-0.25mm}\norm*{\alpha^*}_1\hspace{-0.25mm}) \hspace{-0.25mm} + \hspace{-0.25mm}(\norm*{\mathcal{G}}_{\infty}+1)\hspace{-0.25mm}\sqrt{b}\notag\\
			&\qquad\qquad\quad+w^*(1+K_w)\norm*{\alpha^*}_1\big),\label{err_om}
		\end{align}
		for all $k\in\mathbb{Z}_{[0,L-1]}$, where $K_\Xi$ and $K_w$ are defined in Remark~\ref{Lipschitz_constants}, $b=J(\alpha^*)-\lambda\max\{\varepsilon^*,w^*\}\norm*{\alpha^*}_2^2$ and $\mathcal{P}^k(K_\Xi)=(K_\Xi)^k + (K_\Xi)^{k-1} + \dots + K_\Xi + 1$.
	\end{theorem}
	\begin{proof}
		The proof follows similar steps as the proof of Theorem~\ref{MIMO_inexact_sim_thm} and is omitted for brevity.
\end{proof}
Theorem \ref{MIMO_inexact_om_thm} shows how an approximate solution to the data-based output-matching control problem can be obtained despite the unknown, but uniformly upper bounded, basis function approximation error and noisy output data.\par
The error bound shown in \eqref{err_om} shares many features to that shown in \eqref{err_sim}. For instance, it increases with increasing $k$ (i.e., for longer matched output sequences). The bound can be conservative as well, especially for $K_\Xi>1$. However, it has the same important qualitative property as the bound in \eqref{err_sim}. In particular, the error bound \eqref{err_om} goes to zero as $\max\{\varepsilon^*,w^*\}\to~0$ and if $\{\hat{\Psi}_k(\mathbf{u},\tilde{\Xi})\}_{k=0}^{N-1}$ is persistently exciting of order $L+n$. This can be shown using similar arguments as in Lemma \ref{lemma_eps_zero}.
	\section{Example}\label{examples}
In this section, we illustrate the results of Theorem \ref{MIMO_inexact_sim_thm} on a discretized model of a fully-actuated double inverted pendulum when using an inexact basis function approximation of the unknown nonlinearities and noisy output data (see \eqref{noisy_basis}). Using Euler's discretization of the continuous-time dynamics, we obtain the following discrete-time model\footnote{Although feedback linearization may in general be destroyed under discretization \cite{Grizzle88}, this is not the case for the system \eqref{dt_not_bruno}.}
\begin{align}
	&\mathbf{x}_{k+1} = \mathbf{x}_k + T_s\left(\mathcal{A}\mathbf{x}_k + \mathcal{B}Z_k\right),\quad \mathbf{y}_k = \mathcal{C} \mathbf{x}_k,\label{dt_not_bruno}\\
	&\mathbf{x}_k\coloneqq \begin{bmatrix}\theta_{1,k} & \vartheta_{1,k} & \theta_{2,k} & \vartheta_{2,k}\end{bmatrix}^\top, \quad \mathbf{y}_k=\begin{bmatrix}\theta_{1,k} & \theta_{2,k}\end{bmatrix}^\top,\notag\\
	&Z_k \hspace{-1mm}\coloneqq\hspace{-1mm}\begin{bmatrix}z_{1,k}\\ z_{2,k}\end{bmatrix} \hspace{-1mm}=\hspace{-1mm} M(\theta_k)^{-1}\hspace{-1mm}\left(\tau_k \hspace{-0.5mm} -\hspace{-0.5mm} C(\theta_k,\vartheta_k)\vartheta_k \hspace{-0.75mm}-\hspace{-0.75mm} G(\theta_k)\hspace{-0.25mm}\right)\hspace{-0.5mm}.
	\label{linearizing_controller_ct}
\end{align}
\indent In \eqref{dt_not_bruno}-\eqref{linearizing_controller_ct}, $\theta_k,\,\vartheta_k$ are the vectors of angular positions and angular velocities, respectively, $\tau_k$ is the vector of joint torques and $T_s$ is the sampling time. The terms $M(\theta),\,C(\theta,\vartheta)$ and $G(\theta)$ represent the inertia, dissipative and gravitational terms, respectively, and depend on the masses and lengths of the two links\footnote{The following (unknown) model parameters were used $m_1=m_2=1$kg and $l_1=l_2=0.5$m for the masses and lengths of the two links, respectively.} $m_1,m_2,l_1,l_2$ \cite{spong20}. The outputs in \eqref{dt_not_bruno} have relative degrees $d_1=d_2=2$ and $\sum_id_i=4=n$. Thus, by Theorem \ref{MIMO_FL_thm}, there exists a coordinate transformation $\Xi_k=T(\mathbf{x}_k)$ such that the transformed system is full-state feedback linearizable. This transformation takes the form (by prior model knowledge)
\begin{equation}
	\Xi_k = \begin{bmatrix}
	x_{1,k} & x_{1,k}+T_sx_{2,k} & x_{3,k} & x_{3,k} + T_sx_{4,k}
\end{bmatrix}^\top.\label{coordinate_transformation}
\end{equation}
Re-writing \eqref{dt_not_bruno} in the transformed coordinates results in
\begin{gather}
	\begin{matrix}
		\Xi_{k+1} = \mathcal{A}\Xi_k + \mathcal{B}\mathbf{v}_k,&\qquad 
		\mathbf{y}_k=\mathcal{C}\Xi_k,
	\end{matrix}\label{dt_bruno}\\
	\mathbf{v}_k \coloneqq \begin{bmatrix}
		v_{1,k} \\ v_{2,k}
	\end{bmatrix} = \begin{bmatrix}
		2\xi_{2,k} - \xi_{1,k} + T_s^2z_{1,k} \\
		2\xi_{4,k} - \xi_{3,k} + T_s^2z_{2,k}
	\end{bmatrix}.\label{linearizing_controller_dt}
\end{gather}
To approximate \eqref{linearizing_controller_dt}, we use the following basis functions
\begin{align}
	&\Psi(\tau_k,\Xi_k) =\label{choice_of_Psi}\\ 
	&{\tilde{M}\hspace{-0.5mm}\left(\hspace{-0.5mm}\begin{bmatrix}
			\xi_{1,k}\\ \xi_{3,k}
		\end{bmatrix}\hspace{-0.5mm}\right)}^{\hspace{-0.5mm}-1}\hspace{-1mm}\left(\hspace{-0.5mm} \tau_k \hspace{-0.5mm}- \hspace{-0.5mm}\tilde{C}(\Xi_k)\hspace{-0.5mm}\begin{bmatrix}
	(\xi_{2,k} - \xi_{1,k}) / T_s\\ (\xi_{4,k} - \xi_{3,k}) / T_s
\end{bmatrix}\hspace{-0.5mm} - \hspace{-0.5mm}\tilde{G}\hspace{-0.5mm}\left(\hspace{-0.5mm}\begin{bmatrix}
\xi_{1,k}\\ \xi_{3,k}
\end{bmatrix}\hspace{-0.5mm}\right)\hspace{-0.5mm}\right)\notag
\end{align}
where $\tilde{M},\,\tilde{C},\,\tilde{G}$ contain user-provided estimates for the parameter values (in this example, they were obtained by randomly perturbing the unknown real values by up to
5\%). This choice of basis functions is justified by the fact that \eqref{linearizing_controller_ct} is ubiquitous in robotics, unlike accurate model parameter values which can be very difficult to obtain.\par%
We collect input-output data of length $N=500$ (or 50 seconds with $T_s=0.1$) by operating the double inverted pendulum using a pre-stabilizing controller in the following compact subset of the input-state space
\begin{equation*}
	\Omega = \left \lbrace (\tau_k,\Xi_k)\in\mathbb{R}^2\times\mathbb{R}^4 ~\left|~ \begin{bmatrix}
		\tau_{\textup{lb}} \\ \Xi_{\textup{lb}}
	\end{bmatrix} \hspace{-0.5mm}\leq \hspace{-0.5mm}\begin{bmatrix}
		\tau_k\\ \Xi_k
	\end{bmatrix} \hspace{-0.5mm}\leq \hspace{-0.5mm}\begin{bmatrix}
		\tau_{\textup{ub}} \\ \Xi_{\textup{ub}}
	\end{bmatrix} \right. \right \rbrace,
\end{equation*}
where the inequalities are defined element-wise and $\tau_{\textup{ub}}= -\tau_{\textup{lb}} = \begin{bmatrix}
	20 & 20
\end{bmatrix}^\top$ Nm, $\Xi_{\textup{ub}}=-\Xi_{\textup{lb}}=\begin{bmatrix}
\pi/2 & \pi/2 & \pi/2 & \pi/2
\end{bmatrix}^\top$. In $\Omega$, the basis functions approximation error is upper bounded by\footnote{This was obtained by gridding $\Omega$ and numerically solving \eqref{def_of_G_mat}.} $\varepsilon^*= 0.7296$ as in Assumption \ref{bounded_err_assmp}. Furthermore, the data is contaminated by a random additive noise sampled from a uniform random distribution $U(-0.01,0.01)$ (i.e., $w^*= 0.01$). Figure \ref{fig_sim_results} shows the approximated output of the system as well as the true simulated trajectory when using $\lambda=0.1$. The error between the true and estimated simulated outputs is $\norm*{\bar{\mathbf{y}}-\hat{\mathbf{y}}}_2 = 0.1278$. It can be seen from the figure that the proposed method in Theorem \ref{MIMO_inexact_sim_thm} yields good results when the offline collected data has sufficient information about the system under consideration. In contrast, the error bound in \eqref{err_sim} can potentially be conservative due to the use of Lipschitz continuity arguments in its derivation. Finally, since Theorem \ref{MIMO_inexact_sim_thm} does not require persistency of excitation, it was observed that collecting longer but not necessarily persistently exciting data yields better results on the cost of increased computational burden to solve \eqref{alpha*}.
\begin{figure}[t]
	\centering\includegraphics[width=0.9\columnwidth]{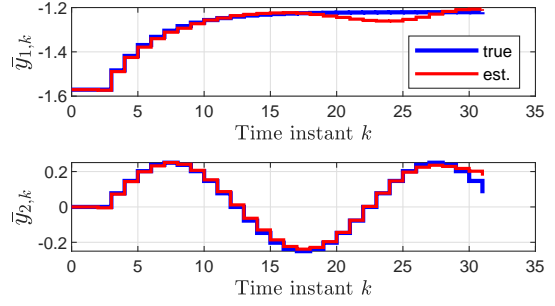}
	\caption{Results of the data-driven simulation problem using the procedure described in Theorem \ref{MIMO_inexact_sim_thm}. The blue lines represent the true outputs and the red lines represent the estimated outputs.}
	\label{fig_sim_results}
\end{figure}
	\section{Conclusions}\label{conclusion_sec}
In this paper, we presented an extension of Willems' fundamental lemma to the class of DT-MIMO feedback linearizable nonlinear systems. This was done by exploiting linearity in transformed coordinates and using a dictionary of known basis functions that depend only on input-output data. In practice, one must account for the non-zero approximation as well as noise in the output data. To that end, we presented constructive methods on how to solve the simulation and output matching problems despite unknown, but uniformly bounded, basis functions approximation errors and output noise, and provided error bounds on the difference between the estimated and actual output trajectories. These error bounds have an appealing qualitative property. In particular, we have shown that the error bounds (and, hence, the actual errors) go to zero if the basis functions approximation error as well as the noise in the data go to zero and persistency of excitation condition is satisfied.

	\bibliographystyle{ieeetr}
	\bibliography{references}
	
	\appendix
\subsection{The block-Brunovsky canonical form}\label{Bruno_app}
	The matrices $\left(\mathcal{A,B,C}\right)$ have the following form
	\begin{equation*}
		\begin{matrix}
			\mathcal{A} \hspace{-0.5mm}\coloneqq\hspace{-1mm} \begin{bmatrix}
			A_1 \hspace{-1mm} & \hspace{-1mm} \dots \hspace{-1mm} & \hspace{-1mm} \mathbf{0}\\
			\vdots \hspace{-1mm} & \hspace{-1mm} \ddots \hspace{-1mm}&\hspace{-1mm} \vdots\\
			\mathbf{0} \hspace{-1mm}& \hspace{-1mm}\dots\hspace{-1mm} & \hspace{-1mm}A_{m}
			\end{bmatrix}\hspace{-1mm},\hspace{-1mm} &\hspace{-1mm} \mathcal{B} \hspace{-0.5mm}\coloneqq \hspace{-1mm}\begin{bmatrix}
			B_1 \hspace{-1mm} & \hspace{-1mm} \dots \hspace{-1mm} & \hspace{-1mm} \mathbf{0}\\
			\vdots \hspace{-1mm} & \hspace{-1mm} \ddots \hspace{-1mm}&\hspace{-1mm} \vdots\\
			\mathbf{0} \hspace{-1mm}& \hspace{-1mm}\dots\hspace{-1mm} & \hspace{-1mm}B_{m}
		\end{bmatrix}\hspace{-1mm},\hspace{-1mm}&\hspace{-1mm}\mathcal{C} \hspace{-0.5mm}\coloneqq\hspace{-1mm} \begin{bmatrix}
				C_1 \hspace{-1mm} & \hspace{-1mm} \dots \hspace{-1mm} & \hspace{-1mm} \mathbf{0}\\
				\vdots \hspace{-1mm} & \hspace{-1mm} \ddots \hspace{-1mm}&\hspace{-1mm} \vdots\\
				\mathbf{0} \hspace{-1mm}& \hspace{-1mm}\dots\hspace{-1mm} & \hspace{-1mm}C_{m}
			\end{bmatrix}
		\end{matrix}\hspace{-0.5mm},
	\end{equation*}
	with $A_i\in\mathbb{R}^{d_i\times d_i},B_i\in\mathbb{R}^{d_i\times 1},C_i\in\mathbb{R}^{1\times d_i}$ for $i\in\mathbb{Z}_{[1,m]}$ defined as
	\begin{equation*}
		\begin{aligned}
			&\begin{matrix}
				A_i \coloneqq \begin{bmatrix}
					0&1&\dots&0 \\ \vdots&\ddots&\ddots&\vdots \\ \vdots& &\ddots&1 \\ 0&\dots&\dots&0
				\end{bmatrix}, &
				B_i \coloneqq \begin{bmatrix}
					0\\ \vdots\\ 0\\ 1
				\end{bmatrix}, &C_i^\top \coloneqq \begin{bmatrix}
				1 \\ 0 \\ \vdots \\ 0
			\end{bmatrix}.
			\end{matrix}
		\end{aligned}
	\end{equation*}
	%%%%%%%%%%%%%
	\subsection{Proof of Lemma \ref{lemma_controllable_pair}}\label{app_controllable_pair}
	A necessary and sufficient condition for controllability of $(\mathcal{A,BG})$ is that the following matrix has full row rank
	\begin{equation}
		\begin{aligned}
			&\begin{bmatrix}\mathcal{BG} & \mathcal{ABG}  & \dots & \mathcal{A}^{n-1}\mathcal{BG} \end{bmatrix}\\
			&= \begin{bmatrix}
				B_1g_1^\top & A_1B_1g_1^\top & \cdots & A_1^{n-1}B_1g_1^\top\\
				\vdots & \vdots & \ddots & \vdots\\
				B_mg_m^\top & A_mB_mg_m^\top & \cdots & A_m^{n-1}B_mg_m^\top
			\end{bmatrix}.\label{ctrb_ABG_step1}
		\end{aligned}
	\end{equation}
	\indent By the structure of $(A_i,B_i)$ in Appendix \ref{Bruno_app}, each block row of \eqref{ctrb_ABG_step1} has the following form
	\begin{equation}
		\begin{aligned}
			&\begin{bmatrix}
				B_ig_i^\top & A_iB_ig_i^\top & \cdots & A_i^{n-1}B_ig_i^\top
			\end{bmatrix}\\[3ex]
			&=\left[
			\begin{array}{ccccc}
				\mathbf{0}& \cdots & \mathbf{0}& g_i^\top & \bovermat{$(\hspace{-0.25mm}n\hspace{-0.5mm}-\hspace{-0.5mm}d_i\hspace{-0.25mm})r\hspace{-0.25mm}$ cols.}{\mathbf{0}\quad \cdots \quad\mathbf{0}} \\
				\mathbf{0}& \cdots & g_i^\top & \mathbf{0}& \mathbf{0}\quad \cdots \quad \mathbf{0}\\
				\vdots & \scalebox{-1}[1]{$\ddots$} & \vdots & \vdots & \vdots \quad \cdots \quad \vdots \\
				g_i^\top & \cdots & \mathbf{0}& \mathbf{0}& \mathbf{0}\quad \cdots \quad \mathbf{0}\\
			\end{array}
			\right],
		\end{aligned}\label{ctrb_ABG_step2}
	\end{equation}
	which clearly has full row rank (i.e., rank $=d_i$) since $\mathcal{G}$ is full row rank and, hence, no row $g_i^\top$ is all zeros. Since, again, $\mathcal{G}$ is full row rank, then the concatenation of the block rows of the form in \eqref{ctrb_ABG_step2} results in a full row rank matrix \eqref{ctrb_ABG_step1} (i.e., rank $=\sum_id_i=n$), implying that the pair $(\mathcal{A,BG})$ is controllable.
	
\end{document}